\documentclass[reqno]{amsart}
\pdfoutput=1

\usepackage{amsmath}
\usepackage{mathrsfs}
\usepackage{amssymb}
\usepackage{amsthm}
\usepackage{hyperref}
\usepackage[top=95pt,bottom=85pt,left=85pt,right=85pt,footskip=45pt]{geometry}
\usepackage[all]{xy}
\xyoption{pdf}
\usepackage{fancyhdr,floatpag}
\floatpagestyle{fancy}

\numberwithin{equation}{section}

\theoremstyle{definition}
\newtheorem*{Def}{Definition}
\newtheorem{Thm}{Theorem}[section]
\newtheorem*{Main_Thm}{Main Theorem}
\newtheorem{Prop}[Thm]{Proposition}
\newtheorem{Lem}[Thm]{Lemma}
\newtheorem{Rmk}[Thm]{Remark}

\newcommand{\PP}{\mathbb{P}}
\newcommand{\CC}{\mathbb{C}}
\newcommand{\QQ}{\mathbb{Q}}
\newcommand{\ZZ}{\mathbb{Z}}
\newcommand{\Oo}{\mathcal{O}}
\newcommand{\Cc}{\mathcal{C}}
\newcommand{\Hom}{\operatorname{Hom}}
\newcommand{\rk}{\operatorname{rank}}
\newcommand{\bb}{\mathfrak{b}}
\renewcommand{\aa}{\mathfrak{a}}

\begin{document}

\title{
Bigness of the tangent bundles\\
of projective bundles over curves
}

\author{Jeong-Seop Kim}

\address{Department of Mathematical Sciences, KAIST, 291 Daehak-ro, Yuseong-gu, Daejeon, 34141 Korea}

\email{jeongseop@kaist.ac.kr}

\begin{abstract}
In this short article, we determine the bigness of the tangent bundle $T_X$ of the projective bundle $X=\PP_C(E)$ associated to a vector bundle $E$ on a smooth projective curve $C$.
\end{abstract}

\maketitle

\section{Introduction}

In this article, all varieties are defined over the field of complex numbers $\CC$.
After Mori's proof of the Hartshorne's conjecture on the ampleness of the tangent bundle $T_X$ \cite{Mor79}, it has been asked to characterize a smooth projective variety $X$ with certain positivity of $T_X$.
For example, a conjecture proposed by Campana and Peternell asks whether the homogeneous varieties are the only smooth Fano varieties $X$ with nef $T_X$, and the conjecture is settled for low dimensions or Picard number one \cite{CP91}, \cite{MOSWW15}.
Recently, a~series of work done by H\"{o}ring, Liu, Shao \cite{HLS20}, and H\"{o}ring, Liu \cite{HL21} investigates smooth Fano varieties $X$ with big $T_X$ as follows.
\begin{Thm}[{\cite{HLS20}, \cite{HL21}}]
Let $X$ be a smooth Fano variety.
\begin{enumerate}
\item[(1)]
If $X$ has dimension $2$, then $T_X$ is big if and only if $(K_X)^2\geq 5$.
\item[(2)]
If $X$ has dimension $3$ and Picard number $1$, then $T_X$ is big if and only if $(K_X)^3\geq 40$.
\item[(3)]
If $X$ has Picard number $1$, and if $X$ contains a rational curve with trivial normal bundle, then $T_X$ is not big unless $X$ is isomorphic to the quintic del Pezzo threefold.
\end{enumerate}
\end{Thm}
The second statement is extended to the following case.
\begin{Thm}[{\cite{KKL22}}]
Let $X$ be a smooth Fano variety of dimension $3$ and Picard number $2$.
Then $T_X$ is big if and only if $(K_X)^3\geq 34$.
\end{Thm}
These results make use of a special divisor on the projective bundle $\PP_X(T_X)$, called the total dual VMRT $\breve\Cc$ (see \cite{HR04}, \cite{OSW16}).
In \cite{HLS20}, they find a formula for $\breve\Cc$, which can be written as follows in the case where $X$ attains a conic bundle structure $X\to Y$.
\[
[\breve\Cc]
\sim \zeta+\Pi^*K_{X/Y}
\]
where $\Pi: \PP_X(T_X)\to X$ is the projection and $\zeta$ is the tautological divisor on $\PP_X(T_X)$.
In other words, $\breve\Cc$~arises as the divisor on $\PP_X(T_X)$ corresponding to the natural subsheaf $T_{X/Y}\to T_X$ of rank $1$.

In this article, we deal with a question on the bigness of $T_X$ in the case of the projective bundle $X=\PP_C(E)$ over a smooth projective curve $C$.
When $E$ has rank $2$, $X$ becomes a ruled surface, and the classification of $X$ with big $T_X$ is a consequence of some known facts.
Indeed, if $E$ is semi-stable, then $h^0(S^k T_X)$ is bounded above by a sum of dimensions of certain family of curves on $X$, whose bound can be obtained from a remark of \cite{Ros02} (see Remark~\ref{semi-stable_rank_two}).
Otherwise, if $E$ is unstable, then the bigness of $T_X$ easily follows from the formula introduced above (cf. \cite[Remark 2.4]{KKL22}).
However, when the rank of $E$ gets larger, we cannot apply the formula because $X\to C$ is not a conic bundle.

In the case of higher ranks, when $E$ is unstable, we can find a rank $1$ subsheaf of $S^m T_X$ instead of $T_X$ to conclude that $T_X$ is big.
Also, when $E$ is semi-stable, by computing an upper bound of $h^0(S^k T_X)$, we~can determine the bigness of $T_X$ according to the stability of $E$ as follows.

\begin{Main_Thm}
Let $C$ be a smooth projective curve and $E$ be a vector bundle on $C$.
Then the projective bundle $X=\PP_C(E)$ has big tangent bundle $T_X$ if and only if $E$ is unstable or $C=\PP^1$.
\end{Main_Thm}

The proof is divided into two parts; the case where $E$ is semi-stable (Theorem \ref{semi-stable_case}) and $E$ is unstable (Theorem \ref{unstable_case}).
The exceptional case is explained in Remark \ref{trivial_case}.
It is worth noting that the result is no longer true for varieties other than curves; there exist stable bundles $E$ of rank $2$ on $\PP^2$ such that one of $E$ gives big $T_X$ whereas another choice of $E$ gives not big $T_X$ for $X=\PP_{\PP^2}(E)$ (see No.~24, 27, and 32 of Table 1 in \cite{KKL22}; No. 24 is the only case with not big $T_X$, and see also \cite{SW90}).

\medskip\noindent
\textbf{Acknowledgement. }
I would like to thank my thesis advisor Prof. Yongnam Lee for suggestion of this problem and valuable comments.
I also thank Chih-Wei Chang for pointing out an error in the previous manuscript.
\medskip

\section{Preliminaries}

Let $X$ be a smooth projective variety of dimension $n>0$ and $V$ be a vector bundle of rank $r>0$ on $X$.
In this article, $\PP_X(V)$ denotes the projective bundle with the projection $\Pi: \PP_X(V)\to X$ in the~sense of Grothendieck.
That is, for the tautological line bundle $\Oo_{\PP_X(V)}(1)$ on $\PP_X(V)$, we have
\[
\Pi_*\Oo_{\PP_X(V)}(m)=
\begin{cases}
S^mV & \text{for $m\geq 0$,}\\
0 & \text{for $m< 0$}
\end{cases}
\]
where the $0$-th power is taken to be $S^0 V=\Oo_X$ for convenience.

For $m\geq -r$ and any vector bundle $W$ on $X$,
\[
R^i\Pi_*(\Pi^*W\otimes\Oo_{\PP_X(V)}(m))
=W\otimes R^i\Pi_*\Oo_{\PP_X(V)}(m)
=0
\ \ \text{for all $i>0$}.
\]
Thus, when $m\geq -r$,
\[
H^i(\PP_X(V),\Pi^*W\otimes\Oo_{\PP_X(V)}(m))
\cong H^i(X,W\otimes\Pi_*\Oo_{\PP_X(V)}(m))
\ \ \text{for all $i\geq 0$}.
\]
In particular, $H^0(\Pi^*W\otimes\Oo_{\PP_X(V)}(-1))=0$.

\subsection{Bigness of Vector Bundle}
In this article, we define certain positivity of a vector bundle by the same positivity of the tautological line bundle on the projective bundle associated to the given vector bundle.
The definition may differ by articles, for example, there are distinct notions of bigness of vector bundles; $L$-big and $V$-big (see \cite{BKKMSU15}).

\begin{Def}
A vector bundle $V$ is said to be \emph{ample} (resp., \emph{nef}, \emph{big}, \emph{effective}, and \emph{pseudo-effective}) on~$X$ if the tautological line bundle $\Oo_{\PP_X(V)}(1)$ is ample (resp., nef, big, effective, and pseudo-effective) on~$\PP_X(V)$.
\end{Def}

\begin{Rmk}
Recall that a line bundle $L=\Oo_X(D)$ on $X$ is big if and only if it satisfies one of the followings (see \cite[Section 2.2]{Laz04}).
\begin{itemize}
\item $h^0(L^k)\sim k^n$ (which is the maximum possible).
\item $mD\equiv_{\text{num}} A+E$ for some integer $m>0$, ample divisor $A$, and effective divisor $E$ on $X$.
\item $D$ lies in the interior of the closure $\overline{\text{Eff}}(X)\subseteq N^1(X)$ of the cone of effective divisors (as bigness is well-defined under numerical equivalence).
\end{itemize}
In the case of vector bundles, we have the following (see also \cite[Section 6.1]{Laz04}).
\begin{itemize}
\item $V$ is big if and only if $h^0(S^k V)\sim k^{n+r-1}$ (which is the maximum possible).
In particular, $T_X$~is big if and only if $h^0(S^k T_X)\sim k^{2n-1}$.
\end{itemize}
\end{Rmk}

We will denote by $\zeta$ the tautological divisor on $\PP_X(V)$.
By the fact that $\zeta+m\Pi^*A$ is ample for~some integer $m>0$ and ample divisor $A$ on $X$ \cite[Proposition 1.45]{KM98}, the following lemma can be shown.

\begin{Lem}[{\cite[Lemma 2.3]{HLS20}}]\label{peff_plus_big_on_base_is_big}
Let $V$ be a vector bundle on a normal projective variety $X$.
Then $V$ is big if and only if $V\otimes \Oo_X(-D)$ is pseudo-effective for some big $\QQ$-Cartier $\QQ$-divisor $D$ on $X$.
\end{Lem}

Note that $V\otimes \Oo_X(-D)$ is pseudo-effective if and only if $\Oo_{\PP_X(V)}(1)\otimes\Pi^*\Oo_X(-D)$ is pseudo-effective, and it is equivalent to saying that $\Oo_{\PP_X(V)}(k)\otimes\Pi^*\Oo_X(-kD)$ is pseudo-effective for some $k>0$.
Thus~if~$S^kV\otimes\Oo_X(-D')$ is effective for some big divisor $D'$ on $X$, then we can say that $V$ is big by the lemma.
As an application of the lemma, we present a proof of the following fact.

\begin{Prop}\label{bigness_of_product}
Let $X$ and $Y$ be smooth projective varieties with big tangent bundles $T_X$ and $T_Y$.
Then the tangent bundle $T_{X\times Y}$ of $X\times Y$ is big.
\end{Prop}

\begin{proof}
Let $D$ and $E$ be big and effective divisors on $X$ and $Y$ respectively.
As $T_X$ and $T_Y$ are big, there exist integers $m,\,n>0$ such that $S^m T_X(-D)$ and $S^n T_Y(-E)$ are effective by Kodaira's Lemma.
Note that $T_{X\times Y}=p^*T_X\oplus q^*T_Y$ for the natural projections $p:X\times Y\to X$ and $q:X\times Y\to Y$, and $p^*D+q^*E$ is a big divisor on $X\times Y$.
Since $S^{m+n}T_{X\times Y}$ contains $S^m p^*T_X\otimes S^n q^*T_Y$ as a direct summand, we~have
\[
H^0(S^{m+n}T_{X\times Y}\otimes \Oo_{X\times Y}(-p^*D-q^*E))
\supseteq H^0(S^m p^*T_X\otimes\Oo_{X\times Y}(-p^*D)\otimes S^n q^*T_Y\otimes\Oo_{X\times Y}(-q^*E))
\neq 0.
\]
Thus $S^{m+n}T_{X\times Y}\otimes\Oo_{X\times Y}(-(p^*D+q^*E))$ is effective, and hence $T_{X\times Y}$ is big by Lemma~\ref{peff_plus_big_on_base_is_big}.
\end{proof}

\subsection{Stability of Vector Bundle}
In this article, the stability is defined in the sense of Mumford and Takemoto.
For the definitions introduced in this section, we add a mild condition (torsion-freeness) from the definitions in the reference \cite[Chapter 1]{HL10}.

Let $Y$ be a smooth projective variety and $E$ be a coherent sheaf on $Y$ with $\text{Supp}(E)=Y$.
Then there exists an open dense subset $U\subseteq Y$ such that $E|_U$ is locally free.
The \emph{rank} of $E$ is defined by $\rk E=\rk E|_U$.

\begin{Def}
Fix an ample divisor $H$ on $Y$.
For a coherent sheaf $E$ on $Y$ with $\text{Supp}(E)=Y$, the \emph{$H$-slope} of $E$ is defined by
\[
\mu_H(E)=\frac{\deg_H E}{\rk E}
\]
where the $H$-degree of $E$ is defined by $\deg_H E=c_1(E).H^{n-1}$.

Let $E$ be a torsion-free coherent sheaf of rank $r>0$ on $Y$.
Then $E$ is said to be \emph{$\mu_H$-stable} (resp., \emph{$\mu_H$-semi-stable}) if for every coherent subsheaf $F$ of $E$ with $0<\rk F< r$,
\[
\mu_H(F)<\mu_H(E)
\quad\text{(resp., $\mu(F)\leq \mu(E)$)}.
\]
Also, $E$ is said to be \emph{$\mu_H$-unstable} if it is not $\mu_H$-semi-stable.
If there is no confusion in the choice of~$H$, then we just denote it by $\mu$-stable (resp. $\mu$-semi-stable, $\mu$-unstable), or stable (resp. semi-stable, unstable) in the case where $Y$ is a curve.
\end{Def}

\begin{Rmk}
The followings are some known facts on the $\mu$-stability and slope of torsion-free coherent sheaves $E$ and $F$ on $Y$.
For the proofs, we may refer \cite[Chapter 3]{HL10}. 
\begin{itemize}
\item If $E$ and $F$ are $\mu$-semi-stable and $\mu(E)<\mu(F)$, then $\Hom(F,E)=0$.
\item If $E$ and $F$ are $\mu$-semi-stable, then $E\otimes F$ is $\mu$-semi-stable.
\item If $E$ is $\mu$-semi-stable, then $S^m E$ is $\mu$-semi-stable for all $m>0$.
\item $\rk(S^m E)=\binom{m+r-1}{r-1}$, $c_1(S^m E)=c_1(E)^{\otimes\binom{m+r-1}{r}}$, and $\mu(S^mE)=m\cdot \mu(E)$.
\item Assume that $E$ fits into the following exact sequence of vector bundles on $Y$.
\[
0\to F\to E\to Q\to 0
\]
If $\mu(F)=\mu(E)=\mu(Q)$, then $E$ is $\mu$-semi-stable if and only if both $F$ and $Q$ are $\mu$-semi-stable.
\item When $E$ is a vector bundle, $E$ is $\mu$-semi-stable if and only if its dual $E^\vee$ is $\mu$-semi-stable, and $\mu(E^\vee)=-\mu(E)$.
\end{itemize}
\end{Rmk}

For a torsion-free coherent sheaf $E$ on $Y$, there exists a canonical filtration
\[
0=E_0\subset 
E_1\subset
\cdots\subset
E_k=E,
\]
which satisfies
\begin{itemize}
\item $E_i/E_{i-1}$ is $\mu$-semi-stable (also, torsion-free) for all $0<i\leq k$, and
\item $\mu(E_{i+1}/E_i)<\mu(E_i/E_{i-1})$ for all $0<i<k$.
\end{itemize}
This filtration is called the \emph{Harder-Narasimhan filtration} of $E$.
We call $F=E_1$ the \emph{maximal destabilizing subsheaf} of $E$.
When $E$ is $\mu$-unstable, we must have $\mu(F)>\mu(E)$.
Also, it follows from the definition that $E/F$ is torsion-free.
Note that, in the case of curves $Y=C$, a coherent sheaf is torsion-free if and only if it is locally free, so we can further say that $E/F$ is locally free.

\section{Semi-Stable Case}

In this section, let $Y$ be a smooth projective variety of dimension $n>0$, and fix an ample divisor $H$ on $Y$.
Let $E$ be a vector bundle of rank $r>0$ on $Y$.
We denote by $X=\PP_Y(E)$ the projective bundle associated to $E$ with the projection $\pi: \PP_Y(E)\to Y$, and by $\Oo_X(\xi)$ the tautological line bundle of $X$.
Then, after taking symmetric powers to the relative Euler sequence
\[
0
\to \Oo_X
\to \pi^*E^\vee\otimes\Oo_X(\xi)
\to T_{X/Y}
\to 0,
\]
we obtain the following exact sequence on $X$.
\begin{equation}\label{symmetric_product_of_relative_Euler_sequence}
0
\to S^{m-1}\pi^*E^\vee\otimes\Oo_X((m-1)\xi)
\to S^m\pi^*E^\vee\otimes\Oo_X(m\xi)
\to S^mT_{X/Y}
\to 0
\end{equation}
By pushing forward the exact sequence via $\pi$, we have the following exact sequence on $Y$.
\[
0
\to S^{m-1}E^\vee\otimes S^{m-1}E
\to S^mE^\vee\otimes S^mE
\to \pi_*S^mT_{X/Y}
\to 0
\]

\begin{Lem}\label{semi-stability_of_push-forward_of_relative_tangent_bundle}
If $E$ is $\mu$-semi-stable, then $\pi_*S^m T_{X/Y}$ is a $\mu$-semi-stable bundle of $\deg_H\pi_*S^mT_{X/Y}=0$ on $Y$ where $X=\PP_Y(E)$ and $\pi: \PP_Y(E)\to Y$ is the projection.
\end{Lem}

\begin{proof}
Note that $S^m E^\vee\otimes S^m E$ is $\mu$-semi-stable for all $m>0$ because $E$ is $\mu$-semi-stable.
Moreover, we have $\deg_H\pi_*S^mT_{X/Y}=0$ due to the above sequence and
\[
\deg_H (S^mE^\vee\otimes S^m E)
=\rk (S^m E)\cdot\deg_H (S^mE^\vee)+\rk (S^m E^\vee)\cdot\deg_H (S^m E)
=0.
\]
Since $\pi_*S^mT_{X/C}$ is a quotient of a $\mu$-semi-stable bundle of the same $H$-slope, it is $\mu$-semi-stable.
\end{proof}

\begin{Prop}\label{semi-stable_general_case}
Assume that $T_Y$ is $\mu$-semi-stable and $\deg_H T_Y<0$.
If $E$ is $\mu$-semi-stable, then the tangent bundle $T_X$ of $X=\PP_Y(E)$ is not big.
\end{Prop}

\begin{proof}
Since the projection $\pi: X\to Y$ is a smooth morphism, there is the following exact sequence of vector bundles on $X$.
\[
0
\to T_{X/Y}
\to T_X
\to \pi^*T_Y
\to 0
\]
From the above sequence, we can bound the dimension of the global sections of $S^k T_X$ as follows.
\[
h^0(S^k T_X)\leq\sum_{m=0}^k h^0(S^mT_{X/Y}\otimes S^{k-m}{\pi^*T_Y}).
\]

By the assumption, $\left(S^{k-m}{T_Y}\right)^\vee$ is $\mu$-semi-stable, and $\deg_H \left(S^{k-m}{T_Y}\right)^\vee>0$ so that
\[
h^0(S^m T_{X/Y}\otimes S^{k-m}{\pi^*T_Y})
=h^0(\pi_* S^m T_{X/Y}\otimes S^{k-m}{T_Y})
=\dim \Hom(\left(S^{k-m}{T_Y}\right)^\vee,\pi_* S^m T_{X/Y})
=0
\]
whenever $0\leq m<k$ due to Lemma \ref{semi-stability_of_push-forward_of_relative_tangent_bundle}.
Thus
\[
h^0(S^k T_X)
\leq h^0(S^k T_{X/Y})
=O(k^{n+r-2}).
\]
That is, $T_X$ is not big.
\end{proof}
\pagebreak

\begin{Thm}\label{semi-stable_case}
Let $C$ be a smooth projective curve of genus $g>0$ and $E$ be a vector bundle on $C$.
If~$E$ is semi-stable, then the tangent bundle $T_X$ of $X=\PP_C(E)$ is not big.
\end{Thm}

\begin{proof}
If $g\geq 2$, then $\pi^*T_C$ is stable as it is a line bundle, thus $T_X$ is not big by Proposition \ref{semi-stable_general_case}.
Otherwise, if $g=1$, then $T_C=\Oo_C$.
According to \cite[Lemma 15]{Ati57}, $h^0(S^mT_{X/C})=h^0(\pi_*S^mT_{X/C})$ is bounded above by the number of indecomposable direct summands of $\pi_*S^mT_{X/C}$ as $\pi_*S^mT_{X/C}$ is a semi-stable bundle of degree $0$ on $C$.
Thus we have
\[
h^0(S^m T_{X/C})
\leq \rk(\pi_*S^mT_{X/C})
=\rk(S^mE^\vee\otimes S^mE)-\rk(S^{m-1}E^\vee\otimes S^{m-1}E).
\]
After telescoping, we can conclude that
\[
h^0(S^k T_X)
\leq\sum_{m=0}^k h^0(S^mT_{X/C})
\leq\rk(S^k E^\vee\otimes S^k E)
=O(k^{2r-2}).
\]
That is, $T_X$ is not big as $X$ has dimension $r$.
\end{proof}

\begin{Rmk}\label{semi-stable_rank_two}
Let $E$ be a semi-stable bundle of rank $2$ on a smooth projective curve $C$ of genus $g>0$.
Then $T_{X/C}$ is a line bundle on $X$, and
\[
S^mT_{X/C}
={T_{X/C}}^{\otimes m}
\cong\Oo_X(2mC_0)
\]
for some $\QQ$-divisor $C_0$ on $X$ with ${C_0}^2=0$.
For a divisor $\bb$ on $C$, we denote by $\Oo_X(\bb f)=\pi^*\Oo_C(\bb)$.

If~$\deg \bb<0$, then $h^0(\Oo_X(2mC_0+\bb f))=0$ because there is no effective divisor $D$ on $X$ with $D^2<0$ (cf. \cite[Section 1.5.A]{Laz04}).
If $\deg \bb=0$, then it is known from \cite[Remark in p.~122]{Ros02} that $h^0(\Oo_X(2mC_0+\bb f))\leq 1$ whenever there is an integral effective divisor $D\sim 2mC_0+\bb f$ for some $m>0$.\linebreak
Due to the remark, if such $D$ is not integral, then it must given by a multiple of divisors numerically equivalent to $C_0$.
In this case, we can find an upper bound of the dimension of the family of such $D$ as $E$ splits once we have $h^0(\Oo_X(C_0+\aa f))\geq 2$ for any divisor $\aa$ of degree $0$ on $C$ \cite[Lemma 5.4]{NR69}.
\end{Rmk}

\begin{Rmk}\label{trivial_case}
If $g=0$ and $E$ is semi-stable, then $C=\PP^1$ and $E=\Oo_{\PP^1}(a)^{\oplus r}$ for some $a\in\ZZ$.
Thus $X\cong \PP^{r-1}\times\PP^1$ and $T_X$ is big by Lemma \ref{bigness_of_product}.
\end{Rmk}

\begin{Rmk}
Using the result on curves, we can state the non-bigness of $T_X$ without the $\mu$-semi-stability of $T_Y$ under some special assumptions on $Y$ and $E$.
Assume that $Y$ has a fibration $p:Y\to B$ with a smooth base $B$ and general fiber $f$ being a smooth curve of genus $g>0$.
If $E|_f$ is semi-stable on a general fiber $f$, then the tangent bundle $T_X$ of $X=\PP_Y(E)$ is big.

Suppose that $T_X$ is big.
Let $Z=\PP_f(E|_f)$ and $\pi_f:Z\to f$ be the induced projection.
Then, for general $Z=\PP_f(E|_f)$, $T_X|_Z$ is big, so $T_Z$ is necessarily big.
Indeed, from the exact sequence
\[
0
\to T_Z
\to T_X|_Z
\to N_{Z|X}
\to 0,
\]
we have $N_{Z|X}\cong {\pi_f}^*N_{f|Y}\cong {\pi_f}^*\Oo_f^{\oplus n-1}\cong \Oo_Z^{\oplus n-1}$, and it gives the following bound.
\[
h^0(S^k T_X|_Z)
\leq \sum_{m=0}^k h^0(S^m T_Z\otimes S^{k-m}(\Oo_f^{\oplus n-1}))
= \sum_{m=0}^k O(k^{n-2})\cdot h^0(S^m T_Z)
= O(k^{n-1})\cdot h^0(S^k T_Z)
\]
However, as $E|_f$ is assumed to be semi-stable, $Z=\PP_f(E|_f)$ cannot have big $T_Z$ by Theorem \ref{semi-stable_case}, which is a contradiction.
\end{Rmk}

\section{Unstable Case}

In this section, we concentrate on the case where $Y$ is a smooth projective curve $C$ of genus $g\geq 0$.
We continue to use the notation in the previous section, e.\,g., $E$ denotes a vector bundle on $C$.

\begin{Prop}\label{line_bundle_unstabilization}
If $E$ is unstable, then $S^m E$ is \emph{unstabilized} by a line subbundle for some $m>0$;\linebreak
there exists a line subbundle $L$ of $S^m E$ with $\mu(L)>\mu(S^m E)$.
\end{Prop}

\begin{proof}
Let $F$ be the maximal destabilizing subbundle of $E$.
Then $\mu(F)-\mu(E)>0$ as $E$ is unstable.
Also, the quotient $Q$ is locally free, so we obtain the following exact sequence of vector bundles on $C$.
\[
0
\to F
\to E
\to Q
\to 0
\]
By taking the symmetric powers to the exact sequence,
\[
0
\to S^m F
\to S^m E
\to S^{m-1}E\otimes Q 
\to S^{m-2}E\otimes \wedge^2 Q
\to \cdots
\to S^{m-\rk Q}E\otimes \wedge^{\rk Q} Q
\to 0,
\]
we can observe that $S^m F$ is a subbundle of $S^m E$.
Note that $\mu(S^m F)-\mu(S^m E)=m\cdot (\mu(F)-\mu(E))>0$.

Due to \cite{MS85}, for each $m>0$, there exists a line subbundle $L$ of $S^m F$ satisfying
\[
\mu(S^m F)-\mu(L)
\leq \frac{\rk(S^m F)-\rk(L)}{\rk(S^m F)\cdot\rk(L)}\cdot g
< g.
\]
So we can find a line subbundle $L$ of $S^m F$ such that
\[
\mu(L)-\mu(S^m E)
=\left\{\mu(S^m F)-\mu(S^m E)\right\}-\left\{\mu(S^m F)-\mu(L)\right\}
>m\cdot (\mu(F)-\mu(E))-g>0
\]
by taking $m>0$ large enough.
As $S^m F$ is a subbundle of $S^m E$, $L$ is also a nonzero subbundle of~$S^m E$.
Hence we obtain a line subbundle $L$ of $S^m E$ satisfying $\mu(L)>\mu(S^m E)$ for some $m>0$.
\end{proof}

A direct application of the proposition is the following.

\begin{Lem}\label{bigness_by_degree}
If $\deg E>0$, then $E$ is big on $C$.
Equivalently, if $\deg E>0$, then the tautological line bundle $\Oo_X(\xi)$ is big on $X=\PP_C(E)$.
\end{Lem}

\begin{proof}
Assume that $\deg E>0$.
If $E$ is semi-stable, then $E$ is ample by \cite{Har71} (see also ``Main Claim'' below \cite[Theorem 6.4.15]{Laz04}).
Thus $E$ is big.
Otherwise, if $E$ is unstable, then $S^m E$ has a line subbundle $L\to S^m E$ with $\mu(L)> \mu(S^m E)>0$, so $S^m E\otimes L^{-1}$ is (pseudo-)effective with $\deg L>0$.
Thus $S^m E$ is big by Lemma \ref{peff_plus_big_on_base_is_big}.
That is, $E$ is big.
\end{proof}

The lemma remains true for $\QQ$-twisted vector bundles (see \cite[Section 6.2]{Laz04}).
Let $X=\PP_C(E)$.
Then, for a $\QQ$-divisor $\aa$ on $C$, the $\QQ$-twisted vector bundle $E\langle \aa\rangle$ is big on $C$ if and only if $\Oo_X(\xi+\aa f)$ is big on $X$ if and only if $\deg E\langle \aa\rangle >0$ for $\Oo_X(\aa f)=\pi^*\Oo_C(\aa)$.
Thus we can say that $\Oo_X(m\xi+\bb f)$ is~big on $X$ if~and only if $\deg (S^m E\otimes L)>0$ for $\Oo_X(\bb f)=\pi^*L$ by taking $\bb=m\aa$ for all $m>0$.
\pagebreak 

\begin{Thm}\label{unstable_case}
If $E$ is unstable, then the tangent bundle $T_X$ of $X=\PP_C(E)$ is big.
\end{Thm}

\begin{proof}
As $E^\vee$ is also unstable, there exists an integer $m>0$ such that $S^m E^\vee$ has a line subbundle $L\to S^m E^\vee$ with $\mu(L)>\mu(S^m E^\vee)$ by Proposition \ref{line_bundle_unstabilization}.
By twisting $\Oo_X(m\xi)$ after pulling-back via $\pi$, it gives a nonzero subbundle
\begin{equation}\label{the_subbundle}
\pi^*L\otimes\Oo_X(m\xi)\to \pi^*S^mE^\vee\otimes\Oo_X(m\xi).
\end{equation}
Note that there cannot exist a nonzero morphism $\pi^*L\otimes \Oo_X(m\xi)\to S^{m-1}\pi^*E^\vee\otimes\Oo_X((m-1)\xi)$ as $\pi^*(S^{m-1}E^\vee\otimes L^{-1})\otimes\Oo_X(-\xi)$ never has a global section.
Thus \eqref{the_subbundle} induces a nonzero subsheaf
\begin{equation}\label{the_subsheaf}
\pi^*L\otimes \Oo_X(m\xi)\to S^m T_{X/C}
\end{equation}
via \eqref{symmetric_product_of_relative_Euler_sequence} as follows.
\[\xymatrix@C1pc@R1.7pc{
&&\pi^*L\otimes \Oo_X(m\xi) \ar[d] \ar@{-->}[dr]&&\\
0 \ar[r] &
S^{m-1}\pi^*E^\vee\otimes \Oo_X((m-1)\xi) \ar[r] &
S^m\pi^*E^\vee\otimes \Oo_X(m\xi) \ar[r] &
S^mT_{X/C} \ar[r] &
0
}\]
Because $S^mT_{X/C}$ is a subbundle of $S^m T_X$, \eqref{the_subsheaf} induces a nonzero subsheaf
\[
\pi^*L\otimes \Oo_X(m\xi)\to S^m T_X,
\]
and hence $S^mT_X\otimes \Oo_X(-m\xi-\bb f)$ becomes effective for $\Oo_X(\bb f)=\pi^*L$.
Due to Lemma \ref{bigness_by_degree} and the~argument after the lemma, $\Oo_X(m\xi+\bb f)$ is big on $X$ since $S^m E\otimes L$ has positive degree;
\[
\deg(S^m E\otimes L)
=\rk(S^m E\otimes L)\cdot\mu(S^m E\otimes L)
=\rk(S^m E\otimes L)\cdot(\mu(L)-\mu(S^mE^\vee))
>0.
\]
Thus, by applying Lemma \ref{peff_plus_big_on_base_is_big}, we can conclude that $T_X$ is big.
\end{proof}

\end{document}